\documentclass{amsart}

\usepackage{amssymb}
\usepackage{amsfonts}
\usepackage{amsmath}
\usepackage{color}

\newtheorem{theorem}{Theorem}

\newtheorem{conjecture}[theorem]{Conjecture}
\newtheorem{corollary}[theorem]{Corollary}

\newtheorem{example}[theorem]{Example}

\newtheorem{lemma}[theorem]{Lemma}

\newtheorem{remark}[theorem]{Remark}

\def\eps{\varepsilon}
\def\ZZ{\mathbb{Z}}
\def\NN{\mathbb{N}}

\def\PP{\mathbb{P}}
\def\EE{\mathbb{E}}
\def\F{\mathcal{F}}
\def\E{\mathcal{E}}

\def\un{\underline}

\DeclareMathOperator{\var}{var}

\begin{document}

\title{Exponential law for random subshifts of finite type}
\author{J\'er\^ome Rousseau, Benoit Saussol, Paulo Varandas}

\begin{abstract}
In this paper we study the distribution of hitting times for a class of random dynamical systems.
We prove that for invariant measures with super-polynomial decay of correlations  hitting times to dynamically
defined cylinders satisfy exponential distribution. Similar results are obtained for random expanding maps.
We emphasize that what we establish is a quenched exponential law for hitting times.
\end{abstract}


\address{J\'er\^ome Rousseau and Paulo Varandas, Departamento de Matem\'atica, Universidade Federal da Bahia\\
Av. Ademar de Barros s/n, 40170-110 Salvador, Brazil}
\email{jerome.rousseau@ufba.br} 
\urladdr{http://www.sd.mat.ufba.br/~jerome.rousseau}
\email{paulo.varandas@ufba.br, +557191781187}
\urladdr{http://www.pgmat.ufba.br/varandas}

\address{Benoit Saussol, Universit\'e Europ\'eenne de Bretagne, Universit\'e de Brest, Laboratoire de Math\'ematiques de Bretagne Atlantique CNRS UMR 6205, 6 avenue Victor le Gorgeu, CS93837, F-29238 Brest Cedex 3, France}
\email{benoit.saussol@univ-brest.fr}
\urladdr{http://www.math.univ-brest.fr/perso/benoit.saussol/}
\thanks{This work was partially supported by the ANR Perturbations (ANR-10-BLAN 0106), DynEurBraz, FAPESB and CNPq}
\maketitle

\section{Introduction}
The theory of random dynamical systems has been introduced to obtain more accurate models for the motion of particles or physical phenomena in general. Indeed, instead of iterating the same transformation one can add some random noise or small perturbations, or more generally work with a family of transformation randomly chosen to represent the errors of approximations or observations. One can see the review \cite{kifliu} for an introduction to this theory.

Another theory which has been widely studied in the last few years (e.g. the review \cite{sau}) is the quantitative description of recurrence in deterministic dynamical systems.
More precisely, let $(X,T,\mu)$ be a measure preserving dynamical system, the hitting time of a point $x\in X$ to a set $A$ is defined by 
\[\tau_A(x):=\inf\{k>0, T^kx\in A\},\]
when $x\in A$, we will speak of return time. This theory is interested in the behaviour of $\tau_A(x)$ when $\mu(A)\rightarrow0$.

A first point of view is to study the return time of a point $x$ in its $r$-neighborhood (i.e. $A=B(x,r)$) and its behavior when $r\rightarrow 0$. It has been proved \cite{B,BS, S} that for rapidly mixing systems $\tau_{B(x,r)}(x)\underset{r\rightarrow0}{\sim}  r^{-d_\mu(x)}$ where $d_\mu(x)$ is the pointwise dimension of the measure $\mu$ in $x$. We refer the reader to \cite{G1,G2} for the same type of results for hitting time and \cite{RS,R} for generalizations.

Another point of view is to study the distribution of return times and hitting time statistics (we can cite the review of Coelho \cite{C} and Abadi and Galves \cite{AG} and also the article of Collet, Galves and Schmitt \cite{CGS} which is one of the first results on this domain). More precisely, we define the distribution of normalized hitting time by
\[F^{hit}_A(t)=\mu\left(\left\{x\in X:\tau_A(x)>\frac{t}{\mu(A)}\right\}\right)\]
and the distribution of normalized return times by
\[F^{ret}_A(t)=\frac{1}{\mu(A)}\mu\left(\left\{x\in A:\tau_A(x)>\frac{t}{\mu(A)}\right\}\right).\]
These works studied the convergence in law of the distribution of normalized hitting and return times when $\mu(A)\rightarrow0$ for sets $A$ well-chosen (for example cylinders of a partition).
 
Haydn, Lacroix and Vaienti \cite{HLV} proved that the limit of the distribution of the return times exists if and only if the limit of the distribution of the hitting times exists. Moreover, an exponential distribution was proved for various families of dynamical systems: Axiom A diffeomorphisms \cite{Hir}, Markov chains \cite{P}, some rational transformations \cite{hay}, uniformly expanding transformations of the interval \cite{Co}, and some non-uniformly hyperbolic systems \cite{hsv,v}. Recently, Freitas, Freitas and Todd \cite{FFT1, FFT2} linked hitting time statistics to extreme value theory.
  
Despite the fact that the quantitative study of Poincar\'e recurrence has been widely studied, the quantitative approach to recurrence for random dynamical systems remains much incomplete.
A first attempt was obtained recently by Marie and Rousseau~\cite{MarieR} where they study the random recurrence
rate for super-polynomially mixing random dynamical systems. More precisely, they proved that for rapidly mixing systems, the quenched recurrence rates are equal to the pointwise dimensions of the stationary measure. One can also see the recent article of Ayta\c c, Freitas and Vaienti \cite{AFV} on law of rare events for random dynamical systems.

In this paper, we prove, in Section~\ref{sec:gal} and \ref{sec:shift}, an exponential law for the distribution of the hitting time for random subshifts of finite type assuming some rapid decay of correlations while similar results are proved in Section~\ref{sec:exp} for some random expanding maps.
 Our main theorems are stated precisely in Section \ref{sec:statement} and we apply our result to some random subshift and random expanding maps in Section~\ref{s:examples}.

\section{Statement of the main results}\label{sec:statement}

We first give the definition of a random subshift of finite type.
Let $(\Omega,\theta,\PP)$ be an invertible ergodic measure preserving system, set $X=\NN^\NN$ and let
$\sigma: X \to X$ denote the shift. Let  $b:\Omega \to \mathbb N$ be a random variable such that  $\EE(\log b)<\infty$. 
Let $A=\left\{A(\omega)=(a_{ij}(\omega)):\omega\in \Omega\right\}$ be a random transition matrix, i.e. for any $\omega\in\Omega$, $A(\omega)$ is a $b(\omega)\times b(\theta\omega)$-matrix with entries in $\{0,1\}$, at least one non-zero entry in each row and each column and such that $\omega\mapsto a_{ij}(\omega)$ is measurable for any $i\in\NN$ and $j\in\NN$.  
For any $\omega\in \Omega$
define the subset of the integers $X_\omega=\{1,\ldots,b(\omega)\}$ and
\begin{equation*}
\E_\omega =\{\un{x}=(x_0,x_1,\ldots)\colon x_i\in X_{\theta^i\omega} \text{ and } a_{x_i x_{i+1}}(\theta^i\omega)=1\text{ for all } i\in\NN\}\subset X, 
\end{equation*}
\begin{equation*}
\E = \{(\omega,\un{x})\colon \omega\in\Omega,\un{x}\in\E_\omega\} \subset \Omega\times X.
\end{equation*}
We consider the random dynamical system coded by the skew-product $S : \E \to \E$ given by
$S(\omega,\un{x})= (\theta \omega,\sigma \un{x})$. Let $\nu$ be an $S$-invariant probability
measure with marginal $\PP$ on $\Omega$ and let $(\mu_\omega)_\omega$ denote
its decomposition on $\E_\omega$, that is, $d\nu(\omega,\un{x})=d\mu_\omega(\un{x})d\PP(\omega)$. 
The measures $\mu_\omega$ are called the \emph{sample measures}. We denote by $\mu=\int \mu_\omega \, d\PP$ the marginal of $\nu$ on $X$.

For $\un{y}\in X$ we denote by $C^n(\un{y})=\{\un{z} \in X : y_i=z_i \text{ for all } 
0\le i\le n-1\} $ the  \emph{$n$-cylinder} that contains $\un{y}$. Set $\F_0^n$ as the sigma-algebra in $X$ 
generated by all the $n$-cylinders.

Our hypothesis on $b$ guarantees that the metric entropy $h_\nu(S,\Omega\times\F_0^1)$ is finite and we will denote it by $h$.

We assume the following: there are constants $h\geq h_0>0$, $c>0$, a random variable $C\in L^p(\Omega,\PP)$ 
for some $p\in(0,1]$, a constant $q>\frac{h}{h_0}(1+\frac{3}{p})$ and a function $\alpha(g)$ satisfying $\alpha(g)g^q\to0$ when $g\to+\infty$ such that
for all $m,n$, $A\in\F_0^n$ and $B\in \F_0^{m}$:

\begin{itemize}
\item[(I)] (polynomial decay of correlations) the marginal measure $\mu$ satisfies
\[
\left|\mu(A\cap\sigma^{-g-n}B) -\mu(A)\mu(B)\right|\le C_0\alpha(g);
\]
\item[(II)] (exponential small cylinders) $\mu_\omega(C^n(\un{y}))\le c e^{-h_0 n}$ for any $\un{y}\in X$ and $n\ge 1$, for $\PP$-almost every $\omega\in\Omega$;
\item[(III)] (fibered polynomial decay of correlations) 
\[
\left|\mu_\omega(A\cap\sigma^{-g-n}B) -\mu_\omega(A)\mu_{\theta^{n+g}\omega}(B)\right|\le C(\omega)\alpha(g)
\]
for $\PP$-almost every $\omega\in\Omega$.
\end{itemize}

Given $A\subset X$ consider the hitting time $R(\un{x},A)=\inf\{k\ge 1\colon \sigma^k\un{x}\in A\}$ and set 
$R_n(\un{x},\un{y})=R(\un{x},C^n(\un{y}))$ for $\un{x}, \un{y}\in X$. 

\begin{theorem}\label{thm:main1} 
We assume that hypothesis (I), (II) and (III) hold.
For $\mu$-almost every $\un{y}$, $\PP$-almost every $\omega$ and all $t\ge 0$ we have 
\begin{equation}\label{eq:cvinlaw}
\mu_\omega\left(\un{x}\in X\colon R_n(\un{x},\un{y})>\frac{t}{\mu(C^n(\un{y}))}\right)\to e^{-t},
	\; \text{ as } n\to\infty.
\end{equation}
\end{theorem}

This can be view as a \emph{quenched} exponential law for hitting time. We provide some applications in Section~\ref{s:examples}, while a similar result is obtained for random endomorphisms with some rapidly mixing conditions in Section~\ref{sec:exp}.

The later convergence together with integration over $\Omega$ and dominated convergence theorem yields the
following \emph{annealed} version:

\begin{corollary} 
Under the same hypothesis, for $\mu$-almost every $\un{y}$ and $ t\ge0$,
\begin{equation*}\label{eq:cvinlaw2}
\mu\left(\un{x}\in X\colon R_n(\un{x},\un{y})>\frac{t}{\mu(C^n(\un{y}))}\right)\to e^{-t},
	\; \text{ as } n\to\infty.
\end{equation*}
\end{corollary}

It is natural to conjecture that the convergence in distribution in the theorem holds almost everywhere 
with respect to  the measure $\nu$, that is:

\begin{conjecture}
For $\nu$-a.e. $(\omega,\un{y})$ the convergence \eqref{eq:cvinlaw} in the theorem holds.
\end{conjecture}

Despite the strong similarity of the conjecture with the theorem, these two statements are not comparable. 
In particular, the corollary would not follow from the conjecture, since even in the simple case of a random Bernoulli measure the sample measures $\mu_\omega$ and $\mu$ could well be mutually singular; see Example~\ref{ex12} for details.

We now provide a similar result for a class of maps satisfying some decay of correlations reminiscent of expanding maps.
Let $(\Omega,\theta,\PP)$ be an invertible ergodic measure preserving transformation, 
$X_\omega\subset X$ be subsets of a compact metric space $X$, 
let $f_\omega: X_\omega \to X_{\theta(\omega)}$ be a bimeasurable map and consider the associated random 
dynamical system described by the skew-product $S : \E \to \E$ given by
$ S(\omega,x)=(\theta(\omega), f_\omega(x))$.
As before, let $\nu$ be an $S$-invariant probability measure with marginal $\PP$ on $\Omega$, 
let $(\mu_\omega)_\omega$ denote its decomposition  and let $\mu=\int \mu_\omega \, d\PP$ be the marginal 
of $\nu$ on $X$. Given $k\ge 1$ and $\omega\in\Omega$ we shall use the notation
$f_\omega^k:=f_{\theta^{k-1}(\omega)} \circ \dots \circ f_{\theta(\omega)}\circ f_\omega$. Moreover,
given a measurable deterministic set $A\subset X$
we write $\tau_A^\omega(x)$ the first $k$ such that $f_\omega^k(x)\in A$.

Replace (I) by (I'): There exists $\gamma(\ell)$ going to zero faster than any power of $\ell$ such that: for $\varphi$ Lipschitz on $X$ and $\psi$ measurable bounded on $X$
\[
\left|\int_\Omega\int_X \varphi \, (\psi\circ f_\omega^\ell) d\mu_\omega d\PP(\omega)-\int_X\varphi d\mu\int_X\psi 
d\mu\right|
\le C_0 \gamma(\ell) Lip(\varphi)\sup|\psi|.
\]

Replace assumption (II) by (II') : it exists $d_0>0$ such that 
$$
\underline{d}_\mu(x):=\liminf_{r\to 0} \frac{\log \mu(B(x,r))}{\log r} >d_0
$$
$\nu$-almost everywhere and $\mu_\omega(B(x,r))\le c r^{d_0}$ for all $x,r,\omega$.

Replace assumption (III) by (III'): for $\varphi$ Lipschitz on $X$ and $\psi$ measurable bounded on $X$
\[
\left|\int_X \varphi \, (\psi\circ f_\omega^\ell) \, d\mu_\omega-\int_X\varphi \, d\mu_\omega\int_X\psi \, 
d\mu_{\theta^\ell\omega}\right|
\le C(\omega) \gamma(\ell) Lip(\varphi)\sup|\psi|.
\]

Include assumption (IV'):  there are constants $a,b>0$ such that for all $x$ and $r,\rho>0$ it holds
 $\mu(B(x,r+\rho))\le \mu(B(x,r))+r^{-b}\rho^a$.
 
Include assumption (V'): the system is random-aperiodic, i.e. 
\[\nu\{(\omega,x)\in\E:\exists n\in \mathbb{N},\,f^n_{\omega}x=x\}=0.\]

\begin{theorem}\label{thm:main2}
If the random dynamical system satisfies (I')-(V') then for $\mu$-a.e. $y\in X$, there 
exists random variables $\Delta_r$ defined on $\Omega$ such that $\Delta_r\to0$ in probability and 
\[
\sup_{t\ge0} \left|\mu_\omega\left(x\in X\colon \tau_{B(y,r)}^\omega(x)>t/\mu(B(y,r))\right)-e^{-t}\right| \le \Delta_r(\omega).
\]
\end{theorem}

\begin{remark}
The method does not give the convergence almost surely in $\omega\in \Omega$ as in the previous section.
We recall, however, that the convergence in probability of $\Delta_r$ implies that a.s. $\omega$ there exists a sequence
$r_n^\omega\to0$ such that $\Delta_{r_n^\omega}(\omega)\to0$ as $n\to\infty$.
\end{remark}

The question of the speed of convergence could be aborded in some situations. 
A quite interesting question is also to understand if the presence of exponential law for return times implies
that the fluctuations of  repetition times and empirical entropies do coincide.

\section{Estimates for general random systems and sets}\label{sec:gal}

In this section we describe general results that will be used in the proofs of our main results, and whose strategy follows the line of \cite{hsv}. They are valid for any random dynamical system $f_\omega$ acting on $X$,
where $\theta$ preserves the probability $\PP$ on $\Omega$. Consider
\[
\delta_\omega(A)
	=\sup_{j\ge 1} \left|\mu_\omega(\tau_A^\omega(\cdot)>j)\,\mu_\omega(A) - \mu_\omega(A\cap \{\tau_A^\omega(\cdot)>j\})\right|.
\]
Since $\theta$ is invertible, by $\sigma$-invariance of $\nu$ and almost everywhere uniqueness of the 
decomposition $\nu=\int \mu_\omega \, d\mathbb{P}(\omega) $ we get that the set $$\Omega' = \{\omega\in\Omega\colon \forall i,\ (f_\omega^i)_*\mu_\omega=\mu_{\theta^i\omega}\}$$ has full $\mathbb P$-probability.

\begin{lemma} \label{lem:delta}
For all $\omega\in\Omega'$, integer $k$ and measurable $A\subset X$ we have
\[
\left|\mu_\omega(\tau_A^\omega(\cdot)>k)-\prod_{i=1}^k\left(1-\mu_{\theta^i\omega}(A)\right)\right |
\le
\sum_{i=1}^{k}\delta_{\theta^i\omega}(A)\prod_{j=1}^{i-1} \left(1-\mu_{\theta^j\omega}(A)\right).
\]
\end{lemma}

\begin{proof}
For any integer $i\ge 1$ we have
\[
\begin{split}
\mu_\omega(\tau_A^\omega(\cdot)>i+1)
&=
\mu_\omega(  f_\omega^{-1}(A^c\cap \{\tau_A^{\theta\omega}(\cdot)>i\})) \\
&=
\mu_{\theta\omega}(\tau_A^{\theta\omega}(\cdot)>i)-\mu_{\theta\omega}(A\cap\{\tau_A^{\theta\omega}(\cdot)>i)\}).
\end{split}
\]
Therefore
$
\left|
\mu_\omega(\tau_A^\omega(\cdot)>i+1) - \left(1-\mu_{\theta\omega}(A)\right)\mu_{\theta\omega}(\tau_A^{\theta\omega}(\cdot)>i) 
\right|
\le
\delta_{\theta\omega}(A).
$
An immediate recursive substitution argument finishes the proof of the lemma.
\end{proof}

The proof of Theorem~\ref{thm:main1} is based on the previous lemma. 
The strategy is to prove that the term $\prod_{i=1}^{k}\left(1-\mu_{\theta^i\omega}(A)\right)$  is almost 
surely convergent to $e^{-t}$, and that the error term in the right hand side goes to zero almost surely.

Since the exponential distribution is continuous, the convergence \eqref{eq:cvinlaw} for any $t\ge0$ is equivalent to the convergence for a countable dense set of $t$'s.
Henceforth, to establish the theorems it is sufficient to show that for any $t>0$
we have the convergence $\PP$-almost surely.

Let then $t>0$ be fixed. Given $A\subset X$  let  $k=k_{A,t}=\lfloor t/\mu(A)\rfloor$ 
and define 
\[
M_{A,t}(\omega)=\sum_{i=1}^{k_{A,t}}\mu_{\theta^i\omega}(A).
\]

\begin{lemma}\label{cor:conv}
We have the approximation as $\sup_\omega\mu_\omega(A)\to0$ 
\[
\prod_{i=1}^{k_{A,t}}\left(1-\mu_{\theta^i\omega}(A)\right)  - e^{-M_{A,t}(\omega)} \to 0.
\]
\end{lemma}

\begin{proof}
This result is a consequence of the following simple and instrumental result: 
if $0<\eps\le 1/2$ and $x_1,\ldots,x_k\in[0,\eps]$ then 
\[
\exp\left(-(1+2\eps)\sum_{i=1}^k x_i\right) \le \prod_{i=1}^k(1-x_i) \le \exp\left(-(1-2\eps)\sum_{i=1}^k x_i\right). 
\]
This finishes the proof.
\end{proof}

Observe that by stationarity the expectation of $M_{A,t}$ is
\begin{equation}\label{eq:EMAt}
\EE(M_{A,t})=\int_\Omega \sum_{i=1}^{k_{A,t}}\; \mu_{\theta^i\omega}(A) \; d\PP(\omega) = k_{A,t}\mu(A),
\end{equation}
which by definition of $k_{A,t}$ already shows that $\EE(M_{A,t})\to t$ as $\mu(A)\to0$.

Next, the error term $\sum_{i=1}^k\delta_{\theta^i\omega}(A)$ in Lemma~\ref{lem:delta} decomposes 
as a mixing term and short entrance or return time terms as follows. 
Let $g\le k$ be an integer and set
\[
\begin{split} 
G_{A,k,g}(\omega) &= \sum_{i=1}^k \mu_{\theta^i\omega}(A\cap \{\tau_A^{\theta^i\omega}(\cdot)\le g\}),\\
H_{A,k,g}(\omega) &= \sum_{i=1}^k \sup_{j\ge 1} \left|\mu_{\theta^i\omega}(A\cap (f_\omega^{g})^{-1}\{\tau_A^{\theta^{i+g}\omega}(\cdot)>j\})-\mu_{\theta^i\omega}(A)\mu_{\theta^{i+g}\omega}(\tau_A^{\theta^{i+g}\omega}(\cdot)>j)\right|,\\
K_{A,k,g}(\omega) &= \sum_{i=1}^k \mu_{\theta^i\omega}(A) \, \mu_{\theta^i\omega}(\tau_A^{\theta^i\omega}(\cdot)\le g).
\end{split}
\]

The gap $g$ allows to exploit the mixing assumptions, related to $H_{A,k,g}$, provided that the 
probabilities of hitting or returning into $A$ before time $g$, related to $G_{A,k,g}$ and $K_{A,k,g}$, 
are small since the whole error term is estimated as follows.

\begin{lemma}\label{lem:somme}
For all $\omega\in\Omega'$, any measurable set $A\subset X$ and any integers $g\le k$ we have 
$\displaystyle \sum_{i=1}^k\delta_{\theta^i\omega}(A) \le G_{A,k,g}(\omega)+H_{A,k,g}(\omega)+K_{A,k,g}(\omega)$.
\end{lemma}

\begin{proof} We have
\begin{align*}
\delta_\omega(A) 
&= \sup_{j\ge 1} |\mu_\omega(\tau_A^\omega(\cdot)>j)\mu_\omega(A)-\mu_\omega(A\cap\{\tau_A^\omega(\cdot)>j\})|\\
&\le
\mu_\omega(\tau_A^\omega(\cdot)\le g) \, \mu_\omega(A) +\mu_\omega(A\cap \{\tau_A^\omega(\cdot)\le g\}) \\
& +\sup_{j\ge g} \left| \mu_{\theta^g\omega}(\tau_A^{\theta^g\omega}(\cdot)>j-g) \, \mu_\omega(A)
				-  \mu_\omega(A\cap (f_\omega^{g})^{-1}\{\tau_A^{\theta^g\omega}(\cdot)>j-g\}) \right|.
\end{align*}
Thus the lemma follows by summing up the the previous terms along the finite piece of orbit of $\omega$ by $\theta$. 
\end{proof}

To summarize, to prove that the limiting law is a.s. exponential we are led to prove that $M_{A,t}\to t$ and
$G_A$, $H_A$ and $K_A$ goes to zero as $A$ shrinks to a typical reference point.

\section{Proofs for the random subshifts}\label{sec:shift}

In this section we will prove Theorem~\ref{thm:main1} and the proof now follows the line of \cite{sau}.
Consider the set 
$$
X' = \left\{\un{y} \in X\colon \limsup_{n\to\infty} -\frac{1}{n}\log\mu(C^n(\un{y}))\le h\right\}.
$$
We already noticed that our hypothesis guarantee that the metric entropy $h_\nu(S,\Omega\times\F_0^1)$ is finite.
Therefore, by Shannon-McMillan-Breiman theorem (see \cite{zhu}) we obtain that 
$
\limsup_{n\to\infty} -\frac{1}{n}\log\mu_\omega(C^n(\un{y}))=h
$
for $\nu$-almost every $(\omega,\un{x})$.  Thus it follows from the Jensen's inequality that $\mu(X')=\nu(\Omega\times X')=1$.

We fix some $t>0$ and take $A=C_n(\un{y})$.
For simplicity we denote $M_{A,t}$ by $M_n$ and $k_{A,t}$ by $k_n$. 
We forget also the dependence on $\un{y}$, $g$ and $t$ for the other random variables 
introduced in the previous section and hence we write $G_n$, $H_n$, $K_n$ for notational simplicity.

\begin{lemma}\label{lem:MN}
For all $\un{y} \in X'$ we have $M_n\to t$, $\PP$-almost surely.
\end{lemma}

\begin{proof}
Let 
\[
\Omega_n =\left\{ \omega\in \Omega \colon \sum_{i=1}^{k_n} C(\theta^i\omega)\le k_n^{\frac{3}{p}}\right\}.
\]
We estimate the second moment of $M_n$ on the set $\Omega_n$
$$
\EE(M_n^2 1_{\Omega_n}) 
	= \sum_{i,j=1}^{k_n} \int_{\Omega_n}\mu_{\theta^i\omega}(A)\mu_{\theta^j\omega}(A)d\PP(\omega).
$$

Let $\eps>0$ and consider now $m=m_n=\lfloor e^{h_0n/(1+\eps)}\rfloor$. Near the diagonal, that is when $|i-j|<m$, using hypothesis (II) we have that
\[
\begin{split}
\sum_{|i-i|<m} \int_{\Omega_n} \mu_{\theta^i\omega}(A)\mu_{\theta^j\omega}(A)d\PP(\omega)
&\le
\sum_{|i-j|<m} ce^{-h_0n} \int_\Omega \mu_{\theta^i\omega}(A)d\PP(\omega)\\
&\le
2cm e^{-h_0n} k_n\mu(A)\\
&\le
2ct m e^{-h_0n} .
\end{split}
\]
Far from the diagonal, the independence hypotheses (I) and (III) yield
\[
\begin{split}
\sum_{|i-j|\ge m}\int_{\Omega_n} \mu_{\theta^i\omega}(A)\mu_{\theta^j\omega}(A)d\PP(\omega)
&\le
2\sum_{j\ge i+m}\int_{\Omega_n} C(\theta^i\omega)\alpha(m-n)d\PP(\omega) +\\
&\quad\quad\quad\quad\quad + \int_\Omega\mu_{\theta^i\omega}(A\cap\sigma^{-(j-i)}A)d\PP(\omega)\\
&\le 
2 k_n k_n^{\frac{3}{p}}\alpha(m-n) + 2 \sum_{j\ge i+m} \mu(A\cap\sigma^{-(j-i)}A) \\
&\le
2 k_n^{1+\frac{3}{p}}\alpha(m-n) + k_n^2\mu(A)^2 + k_n^2 C_0 \alpha(m-n).
\end{split}
\]

On the other hand, since the random variable $C\in L^p(\Omega,\mathbb R)$ for some $p\in (0,1]$ it follows 
by Markov inequality that
$$
\PP(\Omega_n^c) 
	= \PP\left(\left(\sum_{i=1}^{k_n} C(\theta^i\omega)\right)^p>k_n^{3}\right)
	\le \PP\left(  \sum_{i=1}^{k_n} C(\theta^i\omega)^p  > k_n^{3}\right)\\
	\le k_n^{-2} \, \EE(C^p).
$$
Thus, we simply have
\[
\EE(M_n^2 1_{\Omega_n^c}) \le \PP(\Omega_n^c) \sup_\Omega M_n^2 \le k_n^{-2}\EE(C^p) k_n^{2} ce^{-h_0n}. 
\]
Combining these estimates with \eqref{eq:EMAt} which gives $\EE(M_n)=k_n\mu(A)$ 
we finally get a control on the variance of $M_n$ 
\[
\begin{split}
\var M_n
&= \EE(M_n^2 1_{\Omega_n}) + \EE(M_n^2 1_{\Omega_n^c}) -\EE(M_n)^2\\
& \le 
2ct m e^{-h_0n} +
2 k_n^{1+\frac{3}{p}}\alpha(m-n) + C_0 k_n^2 \alpha(m-n) +
\EE(C^p) c e^{-h_0n}.
\end{split}
\]
Thus, one can choose $\eps$ small enough such that $\sum_n \var M_n<\infty$. Indeed, for $n$ large enough $k_n^2 \alpha(m-n)<k_n^{1+\frac{3}{p}}\alpha(m-n)<2^{q}t^{1+\frac{3}{p}}e^{n\gamma}$ where $\gamma=(h+\eps)(1+\frac{3}{p})-q\frac{h_0}{1+\eps}$ and by definition of $q$, $\gamma<0$ if $\eps$ is sufficiently small. It is a classical result that any sequence of centered random variables  $(X_n)$ with
$\sum_n \var X_n<\infty$ is such that $X_n\to0$ a.s.\footnote{Since $\sum_n \var(X_n)<\infty$, we can choose a sequence $u_n\to 0$ so that
$\sum_n \var(X_n)/u_n^2<\infty$ also. Since by Chebyshev, $\PP(|X_n|> u_n)\le\var(X_n)/u_n^2$,
we conclude that $|X_n|\le u_n$ eventually a.s. by Borel-Cantelli.
Hence $X_n\to 0$ a.s. See also e.g. \cite[Theorem 22.6]{Bi}.
}
Hence, $M_n-\EE(M_n)\to0$ a.s., from which the conclusion follows since $\EE(M_n)\to t$.
\end{proof}

We now prove that all random variables used in Lemma~\ref{lem:somme} converge to
zero as $n$ tends to infinity. 
We fix a gap of size $g=g_n= \lfloor e^{h_0 n/4}\rfloor$.

\begin{lemma} \label{lem:EGN}
For $\mu$-almost every $\un{y}$ we have $\EE(G_n)\to0$.
\end{lemma}

\begin{proof}
By stationarity of $\mu$ we obtain
\begin{align}\label{eq:meanG}
\EE (G_n )
&= 
\sum_{i=1}^{k_n}\int_\Omega \mu_{\theta^i\omega}(A\cap \{R(\cdot,A)\le g\}) \, d\PP(\omega)  \\
&=
k_n \, \mu(A\cap \{R(\cdot,A)\le g\}) 
=
k_n\, \mu(A) \, \mu(R(\cdot,A)\le g|A), \nonumber
\end{align}
where $\mu(\cdot | A)$ stands for the usual conditional measure on $A$. Using relation \eqref{eq:meanG} above 
and that $k_n\mu(A)\to t$  we are left to prove that $\mu(R(\cdot,A)\le g|A) \to 0$ (as $n$ tends to infinity) for
$\mu$-almost every $\un{y}$.

By the Ornstein-Weiss theorem \cite{ow} we 
have for $\nu$-a.e. $(\omega,\un{x})$
\[
\lim_{n\to\infty} \frac1n \log R_n(\un{x},\un{x})=h>h_0/2>0,
\]
since by our assumptions $h_0 \le h$.
Given $n_0\ge 1$ consider the set $D(n_0) = \{\un{x} \in X\colon R_n(\un{x},\un{x})\ge e^{n h_0/2} \text{ for all } n\ge n_0 \}$.

Let $\eps>0$ be small and fixed.
Since $\mu(D(n_0))=\nu(\Omega\times D(n_0))$ goes to $1$ as $n_0\to\infty$, we can take
$n_0$ so large that $\mu(D(n_0))>1-\eps$. 
Let $\un{y}$ be a Lebesgue density point of $D(n_0)$ for the measure $\mu$. It holds that 
\[
\mu(D(n_0) | C^n(\un{y})) := \frac{1}{\mu(C^n(\un{y}))} \, \mu(C^n(\un{y})\cap D(n_0))\ge 1-\eps
\]
for all large $n$. Therefore 
\[
\mu(R(\cdot,C^n(\un{y}))\le g|C^n(\un{y}))
	\le \mu(D(n_0)^c|C^n(\un{y}))<\eps.
\]
Hence, 
taking a sequence $\eps_q\to0$ gives the conclusion.
\end{proof}

\begin{lemma} \label{lem:GNas}
For $\mu$-a.e $\un{y} \in X$ we have $ G_n\to 0$ $\PP$-almost surely. 
\end{lemma}
\begin{proof}
Reproducing the computation of the second moment in the proof of Lemma~\ref{lem:MN}, with the same $m_n$, taking into account that 
now $A\cap \{R(\cdot,A)\le g\}\in\F_0^{n+g}$ only, we get 
\[
\begin{split}
\var G_n
&=
\sum_{i,j} \int_\Omega \mu_{\theta^i\omega}(A\cap \{R(\cdot,A)\le g\})\mu_{\theta^j\omega}(A\cap \{R(\cdot,A)\le g\})d\PP(\omega) -\EE(G_n)^2\\
&\le
2c \EE(G_n) m e^{-h_0n} +
2 k_n^{1+\frac{3}{p}}\alpha(m-n-g) + C_0 k_n^2 \alpha(m-n-g) +
\EE(C^p) ce^{-h_0n}.
\end{split}
\]
This proves that $\sum_n \var G_n<\infty$. The conclusion follows as in the proof of Lemma~\ref{lem:MN} using Lemma~\ref{lem:EGN}.
\end{proof}

\begin{lemma} \label{lem:HNas}
For all $\un{y}\in X'$ we have $H_n\to0$ $\PP$-almost surely.
\end{lemma}

\begin{proof}
We use the correlation hypothesis (III) to obtain
\[
H_n(\omega)\le  \sum_{i=1}^{k_n} C(\theta^i\omega)\alpha(g-n)
\le
\alpha(g-n) \left(\sum_{i=1}^{k_n} C(\theta^i\omega)^p\right)^{1/p}.
\]
By the ergodic theorem we have in addition $\sum_{i=1}^{k_n} C(\theta^i\omega)^p = O( k_n \mathbb E(C^p))$ for
$\PP$-almost every $\omega$. Consequently
\[
H_n \leq \alpha(g-n) O\left(k_n^{\frac1p}\right) \to 0\quad\PP\text{-almost surely.}
\]
\end{proof}

\begin{lemma} \label{lem:KNas}
For all $\un{y}\in X'$ we have $K_n\to0$, $\PP$-almost surely.
\end{lemma}
\begin{proof} We have $K_n\le gc e^{-h_0n} M_n$ since 
\[
\mu_\omega(R(\cdot,A)\le g)
	=\mu_\omega\Big(\bigcup_{i=1}^g\sigma^{-i}A\Big)
	\le \sum_{i=1}^g \mu_{\theta^i\omega}(A)\le gce^{-h_0n}.
\] 
In addition, $M_n$ converges $\PP$-almost surely by Lemma~\ref{lem:MN}, which gives the conclusion.
\end{proof}

We are now in a position to finish the proof of our first main result.

\begin{proof}[Proof of Theorem~\ref{thm:main1}]
By Lemmas~\ref{lem:GNas},\ref{lem:HNas} and \ref{lem:KNas} we have $G_n+H_n+K_n\to0$ $\PP$-almost surely and for $\mu$-almost every $\un{y}$.
The theorem follows then from Lemma~\ref{lem:delta} using Lemma~\ref{lem:somme}.
\end{proof}

\section{Random endomorphisms with decay of correlations}\label{sec:exp}

This section is devoted to the proof of Theorem~\ref{thm:main2} on random dynamical systems.
Since the strategy is analogous to the one of Section~\ref{sec:shift} we will only write the proofs of the versions of Lemmas \ref{lem:MN}, \ref{lem:EGN} and \ref{lem:HNas} with full details and leave the adaptations of the other lemmas to the reader.

Write $\Gamma(m)=\sum_{\ell=m}^\infty\gamma(\ell)$. Note that $\Gamma$ is also super polynomially decreasing.
Define the set $X'$ by
\[
X'=\{x\in X\colon \overline{d}_\mu(x)\le d_1\}
\]
for some constant $d_1$ sufficiently large. Since the upper dimension is $\mu$-a.e. 
bounded by the dimension of the space $X$ itself, it suffices to take $d_1=\dim(X)$
to get a full measure set.

We fix $t>0$, take $A=B(x,r)$ and set $M_r=M_{A,t}$ analogously to Section~\ref{sec:gal}. 
We will also write $k_r$, $G_r$, $H_r$ and $K_r$ for simplicity.

\begin{lemma}\label{lem:MNadapt}
For all $x \in X'$ we have $M_r\to t$ in probability on $(\Omega,\PP)$ as $r\to 0$ .
\end{lemma}

\begin{proof}
We use the same method as in Lemma \ref{lem:MN}.
Recall that
\[
M_r=\sum_{i=1}^{k_r} \mu_{\theta^i\omega}(B(x,r)).
\]
We use a set $\Omega_r$ defined as $\Omega_n$ with $k_r$ instead of $k_n$, and estimate $E(M_r^21_{\Omega_r})$.
Take $m=m_r=\lfloor r^{-u}\rfloor$, with $u=d_0/2$ and $\rho=\rho_r=r^v$ form some $v>(d_1+b)/a$.

Those $|i-j|< m$ using (II') again give a contribution less than $2c t m r^{d_0}$.

Those $|i-j|\ge m$, using (I') and (III') give a contribution less than (below $\varphi_{x,r}$ denotes 
a $\rho^{-1}$-Lipschitz function on $X$ such that $1_{B(x,r)}\le\varphi_{x,r}\le 1_{B(x,r+\rho)}$)
\begin{align*}
& 
2 \sum_{i=1}^{k_r-m}\sum_{j=i+m}^{k_r} \int_{\Omega_r}\mu_{\theta^i\omega}(B(x,r))\mu_{\theta^j\omega}(B(x,r))d\PP(\omega)\\
&
\le
2\sum_{i=1}^{k_r}\sum_{\ell=m}^{k_r} \int_{\Omega_r}\left[ \int_X\varphi_{x,r}1_{B(x,r)}\circ f_\omega^{\ell} d\mu_{\theta^i\omega}+C(\theta^i\omega)\gamma(\ell)\rho^{-1}\right] d\PP(\omega)\\
&
\le
2\sum_{i=1}^{k_r}\sum_{g=m}^{k_r} \left[C_0\gamma(\ell)\rho^{-1}+\mu(B(x,r+\rho))\mu(B(x,r))\right] + 
\rho^{-1}\sum_{\ell=m}^{k_r}\gamma(\ell) \EE(1_{\Omega_r} \sum_{k=1}^{k_r} C\circ \theta^i) \\
&
\le
C_0 k_r \Gamma(m) \rho^{-1} + k_r^2\mu(B(x,r+\rho))\mu(B(x,r)) + \rho^{-1}\Gamma(m) k_r^{3/p}. \\
\end{align*}
By assumption (IV') we have $\mu(B(x,r+\rho))\le \mu(B(x,r))+r^{-b}\rho^a$ therefore 
the expectation $E(M_r^21_{\Omega_r})\le r^c$ for some constant $c>0$.

On the other hand $E(M_r 1_{\Omega_r^c})$ satisfies the same upper bound that in Lemma \ref{lem:MN}, therefore 
the variance of $M_r$ is bounded from above by $r^c$ (changing the constant $c$ is necessary).

This proves that the variance of $M_r$ converges to zero as $r\to0$, hence $M_r$ itself converges to $t$
in $L^2$, thus in probability.
\end{proof}

\begin{remark}\label{rem:an}
Indeed, since the variance of $M_r$ is bounded by $r^c$ a Borel-Cantelli argument as in the proof of Lemma~\ref{lem:MN}
shows that $M_{r_k}\to t$ a.s. for the subsequence $r_k=a^k$, for any $a\in(0,1)$.
\end{remark}

We now set the gap to $g_r=\lfloor r^{-d_0/4}\rfloor$.

\begin{lemma} \label{lem:EGNr}
For $\mu$-almost every $x$ we have $\EE(G_r)\to0$.
\end{lemma}

\begin{proof}
The proof follows the one of Lemma \ref{lem:EGN}.
We have by stationarity that 
\[
\begin{split}
\EE(G_r)
&=k_r\int_\Omega \mu_\omega(A \cap \{\tau_A^\omega\le g\})d\PP(\omega)\\
&=k_r\int_{\Omega\times X} 1_{B(x,r)}(y) 1_{\{\tau_{B(x,r)}^\omega\le g\}}(y) d\nu(\omega,y).
\end{split}
\]
Using assumptions (I'), (II') and (V'), the random recurrence rate \cite{MarieR}
gives 

\[
\liminf_{r\to0}\frac{\log\tau_r^\omega(x)}{-\log r}\ge \underline{d}_{\mu}(x)\ge d_0
\]
for $\nu$-a.e. $(\omega,x)$.
Let $r_0>0$. Set $\alpha=d_0/2$ and
\[
Q(r_0,y)=\{\omega\in \Omega\colon \exists r<r_0, \tau_{2r}^\omega(y)<r^{-\alpha}\}.
\]
Denote by $(\nu_y)_{y\in X}$ the decomposition of the measure $\nu$ on $X$, meaning
\[
\nu(E) = \int_X\nu_y(\{\omega\in\Omega\colon (\omega,y)\in E\})d\mu(y), 
\]
for all measurable $E\subset \Omega\times X$.
Let $\eta>0$ and set
\[
E_\eta(r_0)=\{y\in X\colon \nu_y(Q(r_0,y))\le \eta\}.
\]
Let $x$ be a Lebesgue density point of the set $E_\eta(r_0)$ for the measure $\mu$, i.e.
\[
\frac{\mu(B(x,r)\cap E_\eta(r_0))}{\mu(B(x,r))}\to 1
\]
as $r\to0$. Hence there exists $r_1<r_0$ such that for any $r<r_1$
\[
\mu(B(x,r)\cap E_\eta(r_0)^c) \le \eta \mu(B(x,r)).
\]
Let $r<r_1$. Since $g<r^{-\alpha}$ we get
\[
\begin{split}
\EE(G_r)/k_r
&=
\int_{\Omega\times X} 1_{B(x,r)}(y) 1_{\{\tau_{B(x,r)}^\omega(y)\le g\}} d\nu(\omega,y)\\
&\le
\int_{\Omega\times X}1_{B(x,r)}(y) 1_{\{\tau_{2r}^\omega(y)<r^{-\alpha}\}} d\nu(\omega,y)\\
&\le
\int_{\Omega\times X}1_{B(x,r)}(y) 1_{\{Q(r_0,y)\}}(\omega) d\nu(\omega,y)\\
&=
\int_X 1_{B(x,r)}(y) \nu_y(Q(r_0,y))d\mu(y)\\
&\le
\mu(B(x,r)\cap E_\eta(r_0)^c)+\eta \mu(B(x,r)\cap E_\eta(r_0))\\
&\le 2\eta \mu(B(x,r)).
\end{split}
\]
Since $\eta$ is arbitrary and the measure of $E_\eta(r_0)$ can be made arbitrarily close to one, 
this shows that $E(G_r)\to0$ for $\mu$a.e. $x$, since $k_r\mu(B(x,r))\to t$.
\end{proof}

Finally, using Lemma~\ref{lem:EGNr} and Markov's inequality, we obtain that $G_r$ converges to zero in probability. 

\begin{lemma} 
For all $x\in X'$ we have $H_r\to 0$ in probability.
\end{lemma}
\begin{proof}
Using assumption (III'), with the function $\varphi_{x,r}$ introduced in Lemma~\ref{lem:MNadapt} to approximate $B(x,r)$, we get
\[
H_r(\omega) \le \sum_{i=1}^{k_r} C(\theta^i\omega)\gamma(g)\rho^{-1} + 2\sum_{i=1}^{k_r} \mu_{\theta^i\omega}(B(x,r+\rho)\setminus B(x,r)).
\]
The first term is treated as in Lemma~\ref{lem:HNas} to get a.s. convergence to zero. 
For the second one its integration over $\Omega$ gives
\[
2 k_r \mu(B(x,r+\rho)\setminus B(x,r)) \le r^c
\]
for some constant $c>0$.
\end{proof}

The proofs that $K_r$ converge to zero a.s. $\omega$ may be proven exactly as in the previous section
so we do not add the details. This proves Theorem~\ref{thm:main2}.

\begin{remark}
If we strengthen assumptions (I') and (III') to allow functions $\varphi$ which are {\it dynamically Lipschitz}, 
such that $B(x,r)\cap \{\tau_{B(x,r)}\le g_r\}$ may be well approximated by these functions, then the proof of Lemma~\ref{lem:GNas} may be adapted to this setting
and give the a.s. convergence for all sequences $r_n=a^n$ (see also Remark~\ref{rem:an}).
This strategy should work for example for the class of random unimodal maps as studied in~\cite{BBM}.
\end{remark}

\section{Examples}\label{s:examples}

In this section we provide some examples that fulfill the hypotheses of our main theorems. 

\begin{example}\label{exbernougibbs}
Let $s\ge1$ and $(\Omega,\theta)$ be a subshift of finite type on the symbolic space $\{0,1,\ldots,s\}^\ZZ$ 
endowed with the distance $d_\Omega(\omega,\tilde \omega)=\sum_{n\in\mathbb Z} 2^{-|n|}|\omega_i-\tilde \omega_i|$.
Let $\PP$ be a Gibbs measure from a H\"older potential.

Let $b\ge1$ and make the shift $\{0,1,\ldots,b\}^\NN$ a random subshift by putting on it the random Bernoulli measures
constructed as follows. Let $W=(w_{ij})$ be a $s\times b$ stochastic matrix with entries in $(0,1)$ and set $q=\max(w_{ij})$. 
Set $p_j(\omega)=w_{\omega_0,j}$. The random Bernoulli measure 
$\mu_\omega$ is defined by $\mu_\omega([x_0\dots x_n])=p_{x_0}(\omega)p_{x_1}(\theta\omega)\dots p_{x_n}(\theta^n\omega)$.
Since $\mu_\omega$ are Bernoulli measures, 
one can observe easily that for all $m,n$, $A\in\F_0^n$ and $B\in \F_0^{m}$:
\begin{equation}\label{prop1bernougibbs}
\left|\mu_\omega(A\cap\sigma^{-g-n}B) -\mu_\omega(A)\mu_{\theta^{n+g}\omega}(B)\right|=0
\end{equation}
for every $g\geq 1$ and every $\omega\in\Omega$. Thus, property (III) is satisfied. 

Moreover, we obtain that for every cylinder $[x_0\dots x_n]$ and  $\omega\in\Omega$
\[
\mu_\omega([x_0\dots x_n])\le q^{n+1}
\]
for all $n\ge 1$, which implies property (II).

Now we will prove that property (I) holds for the marginal probability measure  
$\mu=\int_\Omega \mu_\omega \,d\mathbb P$. 
The proof explores the mixing properties of the base dynamics 
$\theta:\Omega\to\Omega$. In order to estimate the decay 
for the integrated measure $\mu$ we write, for $m,n$, $A\in\F_0^n$ and $B\in \F_0^{m}$:
\begin{align*}
\left|\mu(A\cap\sigma^{-n-g}B)-\mu(A)\mu(B)\right|
	& =\left|\int_\Omega\mu_\omega(A\cap\sigma^{-n-g}B)d\PP-\int_\Omega\mu_\omega(A)d\PP\int_\Omega\mu_\omega(B)d\PP\right|\\	
	&\leq  \int_\Omega\left|\mu_\omega(A\cap\sigma^{-n-g}B)-\mu_\omega(A)\mu_{\theta^{n+g}\omega}(B)\right|d\PP \\
	& + \left|\int_\Omega\mu_\omega(A)\mu_{\theta^{n+g}\omega}(B)d\PP-\int_\Omega\mu_\omega(A)d\PP\int_\Omega\mu_\omega(B)d\PP\right|
\end{align*}
for all $g\geq1$.
Using \eqref{prop1bernougibbs}, the first term in the right hand side above is null. To control the second term consider 
the partition of $\Omega$ given by sets $U\cap \theta^{-n-g}V$, where $U,V$ are cylinders of rank $n,m$ respectively.
By definition of the Bernoulli measure, the value $\mu_\omega(A)$ for $\omega\in U$ is constant equal to say $\eta_{U,A}$.
We denote analogously by $\eta_{V,B}$ the value taken by $\mu_\omega(B)$ for $\omega \in V$. The Gibbs measure $\PP$ is exponentially $\psi$-mixing 
in the sense that there exists a function $\psi$ such that 
\[
\left|\PP(U\cap \theta^{-n-g}V)-\PP(U)\PP(V)\right| \le 
\psi(g) \PP(U)\PP(V),
\]
with $\psi(g)\to0$ exponentially fast.
Writing 
\[
\int_\Omega \mu_\omega(A)\mu_{\theta^{n+g}\omega}(B)d\PP(\omega)
= \sum_{U,V} \eta_{U,A}\eta_{V,B} \PP(U\cap \theta^{-n-g}V)
\]
and
\[
\int_\Omega \mu_\omega(A)d\PP(\omega) = \sum_U \eta_{U,A} \PP(U),\quad
\int_\Omega \mu_\omega(B)d\PP(\omega) = \sum_V \eta_{V,B} \PP(B)
\]
 
we get that 
\[
\left|\mu(A\cap\sigma^{-n-g}B)-\mu(A)\mu(B)\right|
\le
\sum_{U,V} \PP(U)\PP(V) \psi(g) \le \psi(g),
\]
which decays exponentially fast with $g$ (independently of $m$ and $n$) and property (I) holds.

Therefore, it follows from our results that for $\mu$-almost every $\un{y}$, $\PP$-almost every $\omega$ and all $t\ge 0$ we have 
\[
\mu_\omega\left(\un{x}\in X\colon R_n(\un{x},\un{y})>\frac{t}{\mu(C^n(\un{y}))}\right)\to e^{-t},
	\; \text{ as } n\to\infty.
\]

and that for $\mu$-almost every $y\in X$ and all $t\ge 0$ 
we have 
$$
\mu \left(x\in X\colon R_n(\un{x},\un{y})>\frac{t}{\mu (C^n(\un{y}))}\right)\to e^{-t},
	\; \text{ as } n\to\infty.
$$

\end{example}

The next example shows that the sample measures $\mu_\omega$ and the marginal $\mu$ can be mutually singular for a.e. $\omega$,
as announced in Section~\ref{sec:statement}. This is a special case of Example~\ref{exbernougibbs}.

\begin{example}\label{ex12}
Let $\Omega=\{0,1\}^\ZZ$ with the shift $\theta$ and the Bernoulli measure $\PP$ with weights $(1/2,1/2)$.
Make the shift $\{0,1\}^\NN$ a random shift by putting on it the random Bernoulli measures
constructed as follows. Take $p\in(0,1/2)$ and $q=1-p$. Set $p_0(\omega)=p$ if $\omega_0=0$ and $q$ otherwise,
and $p_1(\omega)=1-p_0(\omega)$. The random Bernoulli measure
$\mu_\omega$ is defined by $\mu_\omega([x_0\dots x_n])=p_{x_0}(\omega)p_{x_1}(\theta\omega)\dots p_{x_n}(\theta^n\omega)$.
Indeed $\mu_\omega([x_0\dots x_n])=p^{k_n(\omega,x)}q^{n-k_n(\omega,x)}$ where $k_n(\omega,x)$ is the 
number of 
$i=0,\dots,n$ such that $\omega_i=x_i$. The marginal measure $\mu$ is 
\[
\mu([x_0\dots x_n]) = \int_\Omega \mu_\omega([x_0\dots x_n]) d\PP(\omega)
=\int p_{x_0}(\omega)d\PP(\omega)\dots \int_\Omega p_{x_n}(\theta^n\omega)d\PP(\omega)
\]
since each $p_{x_k}(\theta^k\omega)$ depends only on $\omega_k$ and the $\omega_k$'s are independent.
Moreover each of these integrals is equal to $(p+q)/2=1/2$. Therefore $\mu$ is the Bernoulli measure on 
$\Omega$ with weights $(1/2,1/2)$.
Next, if the density $\rho_\omega$ of the probability measure $\mu_\omega$ with respect to $\mu$ exists on a set of positive measure then the limit of the ratio
\[
\mu_\omega([x_0\dots x_n])/\mu([x_0\dots x_n]) = p^{k_n(\omega,x)}q^{n-k_n(\omega,x)}2^n
\]
should exists $\mu$-almost everywhere on this set, by Lebesgue differentiation theorem, and should be equal to 
$\rho_\omega(x)$. However, since $k_n$ has increments $0$ or $1$ the only possible limits of the ratio are $0$ or $\infty$ since $p$ and $q$ are not equal to $1/2$. Since the density $\rho_\omega$ cannot have finite nonzero value, therefore it does not exist. 
\end{example}

Note that our results extend to the random dynamical systems context some results obtained for in the deterministic 
setting in~\cite{sau}, in which case we take $\mathbb P$ to be a Dirac measure at a fixed point for $\theta$. 
In particular we obtain applications 
to the thermodynamical formalism of random dynamical systems. Let $S:  \mathcal E\to\mathcal E$ be
as before the skew-product  given by $S(\omega,x)=(\theta(\omega), \sigma(x))$. Given a measurable potential
$\phi:\mathcal E\to\mathbb R$ set $\text{var}_n \phi(\omega)=\sup\{ |\phi(\omega,x) - \phi(\omega, \tilde  x)| :
x_i =\tilde x_i \text{ for all } i< n \}$. If $\phi$ satisfies $\int \sup_x |\phi(\omega,x)| \, d\mathbb P(\omega)<\infty$
and $\text{var}_n \phi(\omega)\le K_\phi(\omega) e^{-\tau n}$ for all $n\ge 1$, for some random variable with 
$\log K_\phi \in L^1(\mathbb P)$, then the variational principle holds
\begin{equation*}
\pi_S(\phi)=\sup_{\eta} \Big\{ h_\eta(S) +\int \phi \, d\eta \Big\},
\end{equation*}
where the supremum is taken over all $S$-invariant probability measures and $\pi_S(\phi)$ denotes
the topological pressure of $S$ with respect to $\phi$ c.f. ~\cite{boggun,kif}.
We say that an $S$-invariant probability measure $\mu$ is an \emph{equilibrium state
for $S$ with respect to $\phi$} if it attains the previous supremum. In addition, we say that a probability 
measure $\mu$ that admits a disintegration $(\mu_\omega)_\omega$ is a \emph{fiber Gibbs measure 
with respect to $\phi$} if there exist random variables $\lambda=\lambda(\omega)$, $C_\phi=C_\phi(\omega)>0$ 
such that $\int \log C_\phi(\omega) \; d\mathbb P(\omega)<\infty$ and
\begin{equation}\label{eq:fiber.Gibbs}
C_\phi(\omega)^{-1} 
	\leq \frac{\mu_\omega( C_n(\underline{x})  ) }
	{\exp \big( -\log \prod_{j=0}^{n-1} \lambda(\theta^j(\omega)) +\sum_{j=0}^{n-1} \phi (S^j(\omega,\underline{x}))  \big)}
	\leq C_\phi(\omega)
\end{equation}
for $\mathbb P$-a.e. $\omega$, every $\underline{x}\in \mathcal E_\omega$ and $n\ge 1$.
In fact, under the previous conditions, it follows from \cite[Theorem~2.1]{kif} that there exists a unique equilibrium
state for $S$ with respect to $\phi$ and that it is a fiber Gibbs measure. Although in general the measure of cylinders 
may decay exponentially to zero but not uniformly in $\omega$ we build an example below where this is not the case.
Let us mention that in most of the known results the thermodynamical formalism follows from a carefull analysis of transfer operators. Given $\omega\in\Omega$ the associated random Perron-Frobenius operator is 
\begin{equation}\label{eq:transferop}
 (\mathcal L^\omega g)(x)= \sum_{S(\omega,y)=x} e^{\phi(\omega,y)} g(y)
 \end{equation}
and, in our context, for every continuous $g: \mathcal E_\omega \to \mathbb R$ it defines a continuous function  
$\mathcal L^\omega g$ on $\E_{\theta(\omega)}$. Set $\mathcal L^{\omega,n}=\mathcal L^{\theta^{n-1}(\omega)}  \circ \dots \circ \mathcal L^{\theta(\omega)}\circ \mathcal L^{\omega}$ for all $n\ge 0$.

Let us mention that these results also hold for the random composition of any finite number of uniformly expanding maps as in Theorem~\ref{thm:main2}.

\begin{example} 
Set $\Omega=\{0,1\}^{\mathbb Z}$ with the distance 
$d_\Omega(\omega,\tilde \omega)=\sum_{n\in\mathbb Z} 2^{-n}|\omega_i-\tilde \omega_i|$.
Let $f_0,f_1: \mathbb T^d \to \mathbb T^d$ be $C^{2}$-smooth expanding maps and $d$ denote the product
metric in $\Omega\times \mathbb T^d$.
Assume that $\mathbb P$ is a Bernoulli measure on $\Omega$ as before and consider the potential 
$\phi: \mathcal E \to \mathbb R$ given by $\phi(\omega,x)=-\log|\det Df_{\omega_0}(x)|$, which is piecewise constant.
By the change of variables formula we obtain that the probability measures $\nu_\omega=Leb$ are conformal, in
the sense that $(\mathcal L^\omega)^*\nu_\omega=\nu_{\theta(\omega)}$. Moreover, it follows from \cite[Theorem~2.1]{kif} that there exists a measurable family of continuous and integrable functions 
$(h_\omega)_\omega$ such that $\mathcal L^\omega h_{\omega} = \lambda_\omega h_{\theta(\omega)}$, 
that $\int h_\omega \,d\nu_\omega=1$ and that $(\mu_\omega)_\omega$ given by $d\mu_\omega=h_\omega d\nu_\omega$ is the unique $S$-invariant probability measure that is an equilibrium state for $S$ with respect 
to $\phi$.

Now, since $f_{\omega}$ is either $f_0$ or $f_1$ (finite number of functions) then the potential $\phi$ is locally constant and the family $(\mathcal{L}_\omega)_\omega$ of transfer operators reduce to finitely
many of them. Consequently, there are uniform constants $L>1$ and $\varepsilon>0$ such that for 
$\mathbb P$-a.e. $\omega\in \Omega$ the cones 
$$
\Lambda^\omega_L
	=\Big\{   g>0 : g(x) \le g(y) \exp( L d(x,y)^\alpha))  \text{ for all } d(x,y)<\varepsilon 	\Big\}
$$
of continuous functions are strictly preserved by the random Perron-Frobenius operator for all positive iterates. More precisely, 
for all $n\ge 1$  one has
$\mathcal L^{\omega,n} (\Lambda^\omega_L) \subset \Lambda^{\theta^n(\omega)}_{2L/3}$
 (c.f. \cite[Equations (4.9) and (4.11)-(4.19)]{kif}). 
Since each function $h_\omega$ belongs to the cone $\Lambda^\omega_L$ of observables and also $\int h_\omega d\nu_\omega=1$ it holds that these have uniform H\"older constants. 

In fact, we use the fact that $h_\omega=\lim_{n\to\infty} \mathcal{L}^{\theta^{-n}\omega,n}1$ and the speed of convergence is exponential, meaning that there exists $a>0$ so that
\begin{equation}\label{kif443}
\|h_\omega-\lim_{n\to\infty} \mathcal{L}^{\theta^{-n}\omega,n}1\| \leq e^{-a n}
\end{equation}
for all $n\ge 1$ large, to prove that $\Phi_\omega = \int \varphi h_\omega d\text{Leb}$ varies H\"older continuously 
with $\omega$.  For completeness let us mention that equation~\eqref{kif443} above corresponds to \cite[Equation~4.43]{kif}  with $C_\omega =C_0$ constant.
Now, if one assumes $\omega$ and $\omega'$ are in the same $-2n, \dots, 2n$ cylinder then 
\begin{align*}
|\Phi_\omega - \Phi_{\omega'}| 
	& \leq  \left(\int \varphi d\text{Leb}\right) \; \sup_X |h_\omega-h_{\omega'}| \\
	& \leq \left(\int \varphi d\text{Leb}\right) \; 
			\left( \sup_X |\mathcal L^{\theta^{-n}\omega,n}1-\mathcal L^{\theta^{-n}\omega',n}1| + 2e^{-a n}\right).
\end{align*}
Since $\theta^{-n}\omega$ and $\theta^{-n}\omega'$ are in the same $[-n..2n]$ cylinder and $\mathcal L^\omega$ is
locally constant then then the first summand in the right hand side above is null, leading to  
$|\Phi_\omega-\Phi_{\omega'}| \leq  2 e^{- a n}$ thus proving the claim that $\omega \mapsto \Phi_\omega$ is H\"older continuous.

Moreover, since each function $h_\omega$ belongs to the cone $\Lambda^\omega_L$ of observables and also $\int h_\omega d\nu_\omega=1$  there
exists a uniform constant $\tilde K\ge 1$ (depending only on $\epsilon$) such that 
$| h_\omega(x) | \le \tilde K e^{L \,diam (M)^\alpha}$ and consequently there exists a uniform constant $K>0$ such that 
$\|h_\omega\|_\infty \le K$ for $\mathbb P$-a.e. $\omega$. 
In consequence,
$$
\mu_\omega (B(x,r)) \leq K Leb (B(x,r)) \leq c r^{d}
$$
and, since $\mu=\int \mu_\omega d\mathbb P$, then $\mu$ is also absolutely continuous with respect to $Leb$ with densitiy bounded by $K$. As a consequence we get for all small $r, \rho$ that
$$
 \mu(B(x,r+\rho))\le c (r+\rho)^d \le \mu(B(x,r))+r^{-1}\rho.
$$
This proves that (II') and (IV') hold.

We are now left to discuss the mixing properties (I') and (III'). The fiber mixing property (III') is a consequence of Theorem~2.2 and Equation~5.19 in~\cite{kif} that there exists a random variable $C$ so that 
\[
\left|\int_X \varphi \, (\psi\circ f_\omega^\ell) \, d\mu_\omega-\int_X\varphi \, d\mu_\omega\int_X\psi \, 
d\mu_{\theta^\ell\omega}\right|
\le C(\omega) \gamma(\ell) Lip(\varphi)\sup|\psi|.
\]
We refer the reader to Lemma~6.3 in \cite{KifLim} for the precise estimates leading to the previous
expression.  Furthermore, since there are finitely many expanding
maps, one can check in Section 4 and Section 5 of \cite{kif} that the random variable $C$ can be taken bounded from above by a uniform constant $C_0$.

Now we will prove that property (I') holds for the marginal probability measure  
$\mu=\int \mu_\omega \,d\mathbb P=\int h_\omega d Leb \,d\mathbb P$. 
The proof explores the mixing properties of $\theta:\Omega\to\Omega$. In order to estimate the decay 
for the integrated measure $\mu=\int \mu_\omega \;d\mathbb P$ we write for all $n\ge 1$
\begin{align}
\Big| \int \varphi &(\psi\circ f_\omega^\ell) \, d\mu - \int \varphi \,d\mu \; \int \psi\,d\mu \Big| \nonumber \\
	& =\left| \int \int \varphi (\psi\circ f_\omega^\ell) \, d\mu_\omega d\PP - \int \varphi \,d\mu_\omega d\PP \; 
		\int \psi\,d\mu_{\theta^\ell(\omega)} d\PP \right| \nonumber \\
	& \leq  \int  \left| \int \varphi (\psi\circ f_\omega^\ell) \, d\mu_\omega 
		-\int \varphi \,d\mu_\omega  \int \psi \,d\mu_{\theta^\ell(\omega)} \right| \; d\PP  \label{eqn1}\\
	& + \left|   \int \left( \int \varphi \,d\mu_\omega  \int \psi \,d\mu_{\theta^\ell(\omega)} \right) d\PP   - \int \varphi \;d\mu \int \psi \; d\mu\right|.  \label{eqn2}
\end{align}
On the one hand, that by (III') the term in equation~\eqref{eqn1} is bounded from above by 
$C_0 \; \gamma(\ell) \sup(\psi) \|\varphi\|_\alpha$. On the other hand, by the exponential decay of correlations for the shift $\theta$ we get, if one considers the  observables $\Phi_\omega=\int \varphi d\mu_\omega
=\int \varphi h_\omega d\text{Leb}$ and  $\Psi_\omega=\int \psi d\mu_\omega=\int \psi h_\omega d\text{Leb}$ then the expression in 
equation~\eqref{eqn2} is such that 
\begin{align*}
\Big| \int \Phi_\omega \,  \Psi_{\theta^\ell(\omega)} d\mathbb P - \int \Phi_\omega \, d\PP \int \Psi_\omega d\PP\Big|
	 & \leq K \sup(\Psi_\omega) \|\Phi_\omega\|_\alpha \gamma(\ell)\\
	 & \leq K \sup(\Psi) \|\Phi\|_\alpha \gamma(\ell)
\end{align*}
for some positive constant $K$ and all $\ell\ge 1$. This proves that condition (I') holds.

Finally, let us prove the random aperiodicity condition (V'). We observe that
\begin{align*}
\nu ( (\omega,x)\in \mathcal{E} \colon \exists n\in\mathbb N \;\; f_\omega^n(x)=x ) 
		& \leq  \sum_{n\ge 1}  \nu ( (\omega,x)\in \mathcal{E} \colon  f_\omega^n(x)=x )  \\
		& \leq \sum_{n\ge 1}  \;\int_\Omega \mu_\omega \Big( x\in X \colon  f_\omega^n(x)=x \Big) \; d\mathbb P 
\end{align*}
and, since $\mu_\omega\ll \text{Leb}$, it is enough to prove that 
$ \text{Leb} ( x\in X \colon  f_\omega^n(x)=x ) =0$ for $\mathbb P$-almost every $\omega$.
This property follows immediately from the fact that $f^n_\omega$ is an expanding map, since the periodic
points of all periods are isolated and thus finite. This proves that $S:\mathcal E \to \mathcal E $ is random aperiodic as claimed.

Therefore, it follows from our Theorem~\ref{thm:main2} that for $\mu$-almost every $\un{y}$ and all $t\ge 0$ 
we have 
$$
\mu_\omega \left(x\in X\colon  \tau^\omega_{B(y,r)}(x) >\frac{t}{\mu (B(y,r))}\right)\overset{\mathbb{P}}{\to} e^{-t},
	\; \text{ as } n\to\infty.
$$
\end{example}

\end{document}